\documentclass{amsart}
\usepackage{amssymb,enumerate}
\usepackage[all]{xy}

\theoremstyle{plain}
\newtheorem{theorem}{Theorem}[section]
\newtheorem{proposition}[theorem]{Proposition}
\newtheorem{lemma}[theorem]{Lemma}
\newtheorem{corollary}[theorem]{Corollary}
\newtheorem{example}[theorem]{Example}
\theoremstyle{definition}
\newtheorem{definition}{Definition}
\theoremstyle{remark}
\newtheorem{remark}[theorem]{Remark}

\newcommand{\Hom}[3]{\operatorname{Hom}_{#1}(#2,#3)}
\newcommand{\tansor}[3]{{#1}\otimes_{#2}{#3}}
\newcommand{\ext}[4]{\operatorname{Ext}^{#1}_{#2}(#3,#4)}
\newcommand{\cd}[3]{\operatorname{cd}_{#1}(#2,#3)}
\newcommand{\cc}[2]{\operatorname{cd}_{#1}(#2)}

\newcommand{\h}[3]{\operatorname{H}^{#1}_{#2}(#3)}
\newcommand{\gam}[2]{\Gamma_{#1}(#2)}
\newcommand{\gh}[4]{\operatorname{H}^{#1}_{#2}(#3,#4)}

\newcommand{\limext}[6]{\underset{#1}{\varinjlim}\operatorname{Ext}^{#2}_{#3}({#4}/{{#5}^#1#4},#6)}
\newcommand{\ann}{\operatorname{Ann}}
\newcommand\supp {\operatorname{Supp}}

\newcommand\ass {\operatorname{Ass}}
\newcommand\att{\operatorname{Att}}

\newcommand{\fgrad}[3]{\operatorname{f-grad}(#1,#2,#3)}
\newcommand\depth{\operatorname{depth}}
\newcommand{\fdepth}[2]{\operatorname{f-depth}(#1,#2)}
\newcommand{\pd }{\operatorname{proj\,dim}}

\newcommand\im{\operatorname{im}}
\newcommand\fa{\mathfrak a}
\newcommand\fb{\mathfrak b}
\newcommand\fc{\mathfrak c}
\newcommand\fm{\mathfrak m}
\newcommand\fp{\mathfrak p}
\newcommand\fq{\mathfrak q}
\newcommand\fr{\mathfrak r}
\newcommand\N{\mathbb N}

\begin{document}

\title[ Filter  sequence and generalized local cohomology  ]{Filter regular sequences and  generalized local cohomology  modules}%
\author{Ali Fathi}
\address{Department of Mathematics\\ Science and Research Branch\\
Islamic Azad University\\ Tehran, Iran.}
\email{alif1387@gmail.com}
\author{Abolfazl Tehranian}
\address{Department of Mathematics\\ Science and Research Branch\\
Islamic Azad University\\ Tehran, Iran.}
\email{tehranian@srbiau.ac.ir}
\author{Hossein Zakeri}
\address{Faculty of Mathematical Sciences and Computer\\
Tarbiat Moallem University\\ Tehran, Iran. }
\email{zakeri@tmu.ac.ir}
%\thanks{The auther}
%\date{*}%
%\dedicatory{}%
%\commby{*}%
% ----------------------------------------------------------------
\begin{abstract}
Let $\fa$, $\fb$ be ideals of a commutative Noetherian ring $R$ and let $M$, $N$  be finite $R$-modules. The concept of an $\fa$-filter grade of $\fb$ on $M$ is introduced and several characterizations   and properties of this notion are given. Then, using the above characterizations, we obtain some results on  generalized local cohomology modules $\gh i{\fa}MN$. In particular, first we determine the least integer $i$ for which $\gh i{\fa}MN$ is not Artinian. Then  we prove that $\gh i{\fa}MN$  is Artinian for all $i\in\mathbb N_0$ if and only if  $\dim{R}/({\fa+\ann M+\ann N})=0$. Also, we establish the Nagel-Schenzel formula for generalized local cohomology modules. Finally, in a certain case,  the set of  attached primes of $\gh i{\fa}MN$ is determined and a comparison between this set and the set of attached primes of $\h i{\fa}N$ is given.
 \end{abstract}
\maketitle
\setcounter{section}{0}
% ----------------------------------------------------------------
\section{ Introduction}
 Throughout this paper, $R$ is a commutative Noetherian ring with nonzero identity, $\fa$, $\fb$ are ideals of $R$ and $M$, $N$ ,$L$ are finite $R$-modules. We will use $\mathbb N$ (respectively $\mathbb N_0$ )  to denote the set of positive (respectively non-negative) integers.
  The generalized  local cohomology  functor was first introduced, in the local case,   by Herzog  \cite{h} and, in the general case, by Bijan-Zadeh \cite{b}. The  $i$-th  generalized local cohomology  functor  $\gh i{\fa}{\boldsymbol \cdot \,}{\boldsymbol \cdot}$ is defined by $$\gh i{\fa}XY=\limext niRX{\fa}Y$$ for all $R$-modules $X$, $Y$ and $i\in\mathbb N_0$. Clearly, this notion is a natural  generalization of the ordinary local cohomology functor \cite{bs}.

    There is a lot of  current interest in the theory of filter regular sequences in commutative algebra; and, in recent years, there have appeared many papers  concerned with the role of these sequences in the theory of local cohomology. In particular case, when one works on a local ring, the concept of a filter regular sequence has been studied in \cite{sv, t}  and has led to some interesting results. We will denote   the supremum of all numbers $n\in\mathbb N_0$ for which there exists an $\fa$-filter regular $M$-sequence of length  $n$ in $\fb$ by $\fgrad {\fa}{\fb}M$. In a local ring $(R, \fm)$, $\fgrad {\fm}{\fa}M$ is known as $\fdepth {\fa}M$. L\"{u} and Tang \cite {lt} proved that
 $$\fdepth {\fa}M=\inf\{i\in\mathbb N_0 | \dim \ext iR{{R}/{\fa}}M>0\}$$
  and that $\fdepth {\fa}M$ is the least integer $i$ such that $\h i{\fa}M$ is not Artinian. As a theorem, we generalize their results and characterize $\fgrad {\fa}{\fb}M$ to  non local cases as follows.
 \begin{align*}
  \fgrad {\fa}{\fb}M&=\inf \{i\in\mathbb N_0 |\,\supp \ext iR{{R}/{\fb}}M\nsubseteq V(\fa)\}\\
   &= \inf \{i\in\mathbb N_0 |\,\supp \h i{\fb}M\nsubseteq\ V(\fa)\},\\
        \fgrad {\fa}{\fb+\ann N}M &=\inf \{i\in\mathbb N_0 |\,\supp  \gh i{\fb}NM\nsubseteq V(\fa)\},
       \end{align*}
and
\begin{align*}
\sup_{A\in\mathcal M} \fgrad {\bigcap_{\fm\in A} \fm}{&\fa+\ann M}N\\
&=\inf\{i\in\mathbb N_0|\, \gh i{\fa}MN \, \textrm{is not Artinian}\}\\
    &=\inf \{i\in\mathbb N_0|\, \supp \gh i{\fa}MN \nsubseteq \max (R) \}\\
   &=\inf \{i\in\mathbb N_0|\, \supp \gh i{\fa}MN \nsubseteq A \, \textrm{ for all } A\in\mathcal M\}\\
    &=\inf \{i\in\mathbb N_0|\, \dim \ext iR{{M}/{\fa M}}{N}>0\},
        \end{align*}
        where $\mathcal M$ is the set of all finite subsets of $\max(R)$.

As an application of this theorem, we show that,  if $n\in\mathbb N$, then $\gh i{\fa}MN$ is Artinian for all $i<n$ if and only if $\gh i{\fa R_\fp}{M_\fp}{N_\fp}$ is Artinian for all $i<n$ and all prime ideals $\fp$. Also, we prove that
 $\gh i{\fa}MN$  is an Artinian $R$-module for all $i\in\mathbb N_0$ if and only if  $\dim{R}/({\fa+\ann M+\ann N})=0$. In particular,
 $\ext iRMN$ has finite length for all $i\in\mathbb N_0$ if and only if $\dim R/(\ann M+\ann N)=0$.

    Let $x_1,\ldots,x_n$  be an $\fa$-filter regular $N$-sequence in $\fa$. Then the formula
  $$\h i{\fa}N=
\begin{cases}
\h i{(x_1,\ldots,x_n)}N   &\text{if } i<n\,,\\
\h {i-n}{\fa}{\h n{(x_1,\ldots,x_n)}N}  &\text{if } i\geq n\,,
\end{cases}$$
 is known as Nagel-Schenzel formula (see \cite{NS} and \cite{khs}).  We generalize the above formula for the generalized local cohomology modules. Indeed, we prove that:
\begin{enumerate}[{\rm(i)}]
    \item\label{sch1} $\gh i{\fa}MN\cong\gh i{(x_1,\ldots,x_n)}MN$ for all $i<n$;\\
    \item\label{sch2} if $\pd M =d $ and $L$ is projective, then
                $$\gh {i+n}{\fa}{\tansor MRL}N\cong\gh i{\fa}M{\gh n{(x_1,\ldots,x_n)}LN}$$
      for all $i\geq d$.
    \end{enumerate}

 Assume that $\bar{R}={R}/({\fa+\ann M+\ann N})$ and that the ideal $\fr$ is the inverse image of the Jacobson radical of $\bar R$ in $R$. If $\bar R$ is  semi local, then, by using the isomorphisms described in (i) and  Theorem \ref{art}, we prove that
\begin{align*}
   \fgrad {\fr}{\fa+\ann M}N&=\inf\{i\in\mathbb N_0|\, \gh i{\fa}MN \, \textrm{is not Artinian}\}\\
    &=\inf\{i\in\mathbb N_0|\, \gh i{\fa}MN\ncong \gh i{\fr}MN\}.
   \end{align*}

Let $(R, \fm)$ be a local ring and $\dim N=n$. Macdonald and Sharp \cite[Theorem 2.2]{ms} show that
$$\att\h n{\fm}N=\{\fp\in\ass N | \dim  R/{\fp}=n \}.$$
As an extension of this result, Dibaei and Yassemi \cite[Theorem A]{dy} proved
$$\att\h n{\fa}N=\{\fp\in\ass N | \operatorname {cd}_{\fa}({R/{\fp}})=n \},$$
where $\operatorname{cd}_{\fa}(M)$ is the greatest  integer $i$ such that $\h i{\fa}M\neq0$. Finally, if $d=\pd M<\infty$, then Gu and Chu \cite[Theorem 2.3]{gc} proved that $\gh{n+d}{\fa}MN$ is Artinian and
 $$\att\gh{n+d}{\fa}MN=\{\fp\in\ass  N | \cd{\fa}M{ R/{\fp}}=n+d\},$$
 where, for an $R$-module $Y$, $\cd{\fa}MY$ is the greatest integer $i$ such that $\gh i{\fa}MY\neq 0$. Notice that $\cd {\fa}MN\leq d+n$ \cite[Lemma 5.1]{b}.  We prove the above result in general case where $R$ is not necessarily  local.
As a corollary we deduce that $\att \gh {n+d}{\fa}MN\subseteq\att \h n{\fa}N$.  Also,we give an example to show that this inclusion  may be strict.
Indeed, our example not only show that the Theorem 2.1 of \cite{ma2} is not true, but it also rejects all of the following conclusions in \cite{ma2}.

  Finally, Let  $\pd M=d<\infty$ and $\dim N=n<\infty$ and $\fb=Ann \h {n}{\fa}N$. We prove that, if $R/\fb$ is a complete semilocal ring, then
$$\att \gh {n+d}{\fa}MN=\supp  {\ext dRMR}\cap\att\h n{\fa}N.$$
In particular, if in addition,  $\pd_{R_\fp}M_\fp=\pd M$ for all $\fp\in\supp M$, then
$$\att \gh {n+d}{\fa}MN=\supp  M\cap\att\h n{\fa}N.$$
% ----------------------------------------------------------------
\section{ Filter regular sequences}
 We say that a sequence $x_1,\ldots,x_n$ of elements of $R$ is an $\fa$-filter regular $M$-sequence, if $x_i\notin \fp$ for all $\fp\in\ass {M}/{(x_1,\ldots,x_{i-1})M}\setminus V(\fa)$ and  for all $i=1,\ldots,n$. In addition,  if $x_1,\ldots,x_n$ belong to $ \fb$, then we say that  $x_1,\ldots,x_n$  is an $\fa$-filter regular $M$-sequence  in $\fb$. Note that $x_1,\ldots,x_n$ is  an $R$-filter regular $M$-sequence if and only if it is a weak $M$-sequence in the sense of \cite[Definition 1.1.1]{bh}.

Some parts of the next elementary proposition are included in \cite[ Proposition 2.2]{NS} in the case where $(R, \fm)$ is local and $\fa=\fm$.
\begin{proposition}\label{kh1}
 Let $x_1,\ldots,x_n$ be a sequence of elements of $R$ and $n\in\mathbb N$.Then the following statements are equivalent.
  \begin{enumerate}[{\rm(i)}]\label{filter}
    \item $x_1,\ldots,x_n$ is an $\fa$-filter regular $M$-sequence.\label{filteri}
    \item $\supp ({(x_1,\ldots,x_{i-1})M:_Mx_i})/{(x_1,\ldots,x_{i-1})M}\subseteq V(\fa)$ for all $i=1,\ldots,n$.\label{filterii}
    \item ${x_1}/{1},\ldots,{x_n}/{1}$ is a weak $M_\fp$-sequence  for all $\fp\in\supp  M\setminus V(\fa)$.\label{filteriii}
     \item $x_1^{\alpha_1},\ldots,x_n^{\alpha_n}$ is an $\fa$-filter regular $M$-sequence for all  positive integers $\alpha_1,\ldots,\alpha_n $.\label{filteriv}
     \item $x_i$  is a weak  $({{M}/{(x_1,\ldots,x_{i-1})M}})/{\gam {\fa}{{M}/{(x_1,\ldots,x_{i-1})M}}}$-sequence for all $i=1,\ldots,n$.\label{filterv}
     \item $(x_1,\ldots,x_{i-1})M:_M {x_i}\subseteq (x_1,\ldots,x_{i-1})M:_M \langle\fa\rangle$ for all $i=1,\ldots,n$, where
     $N :_M \langle\fa\rangle=\{x\in M|\, {\fa}^mx\subseteq N\, \textrm{for some}\,\, m\in\mathbb N \}$ for any submodule $N$ of $M$.\label{filtervi}
  \end{enumerate}
  \end{proposition}
%   $\ass M\setminus V(\fa)=\ass {M}/{\gam {\fa}M}$,  (\ref{filteri}) and (\ref{filterv}) are equivalent.
%   Now we prove (\ref{filteri}),  (\ref{filterii}) and (\ref{filtervi}) are equivalent. For an element $x$ of $R$, $x\notin\bigcup_{\fp\in\ass M\setminus V(\fa)} \fp$ if and only if $\ass  (0:_Mx)=V(Rx)\cap\ass M\subseteq V(\fa)$ or  equivalently  $\supp (0:_Mx)\subseteq V(\fa)$.
 %  %On the other hand $\supp  0:_Mx\subseteq V(\fa)$ if and only if ${\fa}^t(0:_Mx)=0$ for some $t\in \mathbb N$.
 %   It therefore follows that  (\ref{filteri}),  (\ref{filterii}) and (\ref{filtervi}) are equivalent.
%% Next we prove  (\ref{filteriii}) and (\ref{filterii}) are equivalent. Let $x\in R$. Then $ x/1$ is a nonzero divisor on $M_\fp$ for all $\fp\in\supp M\setminus V(\fa)$ if and only if $(0:_Mx)_\fp=0$ for all $\fp\in\supp  M\setminus V(\fa)$ or equivalently $\supp  (0:_Mx)\subseteq V(\fa)$. Therefore (\ref{filterii})$\Leftrightarrow$(\ref{filteriii}).
%  Finally, using \cite[Exercise 1.1.10]{bh} and (\ref{filteri})$\Leftrightarrow$(\ref{filteriii}), it is clear that  (\ref{filteri})$\Leftrightarrow$(\ref{filteriv}).
%\end{proof}

It is clear from definition, that, for a given $n\in \mathbb N$, one can find an $\fa$-filter regular $M$-sequence of length $n$.
 The following theorem characterizes the existence of an $\fa$-filter regular $M$-sequence  of length $n$ in $\fb$.

\begin{theorem}\label{filter1}Let $n\in\mathbb{N}$. Then the following statements  are equivalent.

    \begin{enumerate}[{\rm(i)}]
      \item $\fb$ contains an $\fa$-filter regular $M$-sequence of length $n$.
      \item Any $\fa$-filter regular $M$-sequence  in $\fb$  of length  less  than $n$ can be extended to an $\fa$-filter regular $M$-sequence  of length $n$ in $\fb$.
      \item $\supp \ext iR{{R}/{\fb}}M\subseteq V(\fa)$ for all $i<n$.
      \item If $\supp  N=V(\fb)$, then $\supp \ext iRNM\subseteq V(\fa)$ for all $i<n$.
      \item $\supp \h i{\fb}M\subseteq V(\fa)$ for all $i<n$.
      \item If  $\ann N\subseteq\fb$, then $\supp \gh i{\fb}NM\subseteq V(\fa)$ for all $i<n$.
    \end{enumerate}
    \begin{proof}
     The implications (ii)$\Rightarrow$ (i), (iv)$\Rightarrow$ (iii) and (vi)$\Rightarrow$ (v) are clear.

    (i)$\Rightarrow$ (ii). Assume the contrary  that $x_1,\ldots,x_t$ is an $\fa$-filter regular $M$-sequence  in $\fb$ such that $t<n$ and that it  can not be extended  to an $\fa$-filter regular $M$-sequence  of length $n$ in $\fb$.  Then $\fb\subseteq \fp$ for some $\fp \in \ass   {M}/{(x_1,\ldots,x_t)M}\setminus V(\fa)$. So that $\fb R_\fp\subseteq \fp R_\fp\in \ass _{R_\fp} {M_\fp}/{({x_1}/{1},\ldots,{x_t}/{1})M_\fp}$ . It follows that ${x_1}/{1},\ldots,{x_t}/{1}$ is a maximal $M_\fp$-sequence in $\fb{R_\fp}$, which  is a contradiction in view of the hypothesis, Proposition \ref{filter} and \cite[Theorem 1.2.5]{bh}.

         (i)$\Rightarrow$  (iv) Suppose that $x_1,\ldots,x_n$ is an $\fa$-filter regular $M$-sequence in $\fb$. Let $t\in \mathbb N$  be such that $x_i^t\in\ann N$ for all $i=1,\ldots,n$. By Proposition  \ref{kh1}, for any $\fp\in\supp  M\setminus V(\fa)$, ${x_1^t}/{1},\ldots,{x_n^t}/{1}$ is a weak $M_\fp$-sequence in $\ann_{R_\fp}N_\fp$. So that, for all $i<n$, we have
        $\ext i{R_\fp}{N_\fp}{M_\fp}=0.$
        Therefore (\ref{filteriv}) holds.

    (i)$\Rightarrow$(vi) Suppose that $x_1,\ldots,x_n$ is an $\fa$-filter regular $M$-sequence in $\fb$. For any $\fp\in\supp  M\setminus V(\fa)$, ${x_1}/{1},\ldots,{x_n}/{1}$ is a weak $M_\fp$-sequence in $\fb R_\fp$. So that, by \cite[Proposition 5.5]{b}, $\gh i{\fb R_\fp}{N_\fp}{M_\fp}=0$
      for all $i<n$. This proves the implication (i)$\Rightarrow$(vi).

  Next we prove  the implications (iii)$\Rightarrow$(i) and (v)$\Rightarrow$(i)  by induction on $n$. Let $n=1$. In either cases  $\supp \Hom  R{{R}/{\fb}}M\subseteq V(\fa)$. Therefore (i) holds. Suppose that, for all $i\in\mathbb N_0$, $T^i(\boldsymbol{\cdot})$ is either $\ext iR{{R}/{\fb}}{\boldsymbol\cdot}$ or $\h i{\fb}{\boldsymbol\cdot}$. Assume that $n>1$ and $\supp T^i(M)\subseteq V(\fa)$ for all $i<n$. Then $\fb$ contains   an $\fa$-filter regular $M$-sequence, say $x_1$.  The exact sequences
  $$0\longrightarrow 0:_M{x_1}\longrightarrow M\stackrel{x_1}{\longrightarrow} x_1M\longrightarrow0$$
  and  $$0\longrightarrow{x_1}M\longrightarrow M\longrightarrow{M}/{x_1M} \longrightarrow0$$
  induce the long exact sequences
  $$\cdots\longrightarrow T^i (0:_M{x_1})\longrightarrow T^i(M)\longrightarrow T^i (x_1M)\longrightarrow T^{i+1} (0:_M{x_1})\longrightarrow\cdots$$
  and
  $$\cdots\longrightarrow T^i (x_1M)\longrightarrow T^i (M)\longrightarrow T^i ({M}/{x_1M})\longrightarrow T^{i+1} (x_1M)\longrightarrow\cdots.$$
  Since $\supp   0:_M{x_1}\subseteq V(\fa)$, by  Proposition \ref{filter}, it follows that $\supp   T^i ( 0:_M{x_1})\subseteq V(\fa)$ for all $i\in\N_0$.
   Therefore, using the above long exact sequences, we have
   $\supp  T^i ({{M}/{x_1M}})\subseteq V(\fa)$ for all $i<n-1$. Hence, by  inductive hypothesis,  $\fb$ contains an $\fa$-filter regular ${M}/{x_1M}$-sequence  of length $n-1$ such as $x_2,\ldots,x_n$. This completes the inductive step, since $x_1,\ldots,x_n$ is an $\fa$-filter regular $M$-sequence in $\fb$.
    \end{proof}
\end{theorem}
\begin{remark} \label{propinf}
One may use Theorem \ref{filter1} (iii)$\Rightarrow$(ii) and Proposition \ref{kh1} to see that $\supp {M}/{\fb M}\subseteq V(\fa)$ if and only if, for each $n\in\mathbb N$, there exists an $\fa$-filter regular $M$-sequence  of  length $n$ in $\fb$. Moreover, if $\supp {M}/{\fb M}\nsubseteq V(\fa)$,then it follows from Theorem \ref{filter1} that any two maximal $\fa$-filter regular $M$-sequences   in $\fb$ have the same length. Therefore, we may give the following.
\end{remark}
%\begin{proposition}\label{propinf}
%    There exists an $\fa$-filter regular $M$-sequence  of infinite length in $\fb$ if and only if $\supp {M}/{\fb M}\subseteq V(\fa)$.
%\end{proposition}
%\begin{proof
%    If $\supp {M}/{\fb M}\subseteq V(\fa)$, then,  in view of  Theorem \ref{filter1} (iii)$\Rightarrow$(ii), we can construct an  $\fa$-filter regular $M$-sequence  of infinite length in $\fb$.
 %    Conversely, let $x_1,x_2,\ldots$ be an $\fa$-filter regular $M$-sequence  in $\fb$. Proposition \ref{kh1} yields that
 %    $$\supp ({(x_1,\ldots,x_{i-1})M:_Mx_i})/{(x_1,\ldots,x_{i-1})M}\subseteq V(\fa)$$ for all $i\in \mathbb N$. If $\supp {M}/{\fb M}\nsubseteq V(\fa)$, then $(x_1,\ldots,x_{i-1})M\neq(x_1,\ldots,x_i)M$ for all $i\in\mathbb N$, which is not possible.
%\end{proof}
%  Theorem \ref{filter1} and Remark \ref{propinf} enable us to give the following.
 \begin{definition}\label{filter2}Let $\supp {M}/{\fb M}\nsubseteq V(\fa)$.  Then the common length of all maximal $\fa$-filter regular $M$-sequences  in $\fb$ is denoted by $\fgrad {\fa}{\fb}M$ and is called the $\fa$-filter grade of $\fb$ on $M$. We  set $\fgrad{\fa}{\fb}M=\infty$ whenever $\supp {M}/{\fb M}\subseteq V(\fa)$.
\end{definition}
     Let $(R,\fm)$ be a local ring. Then the  $\fm$-filter grade of $\fb$ on $M$ is called the filter depth of $\fb$ on $M$ and is denoted by $\fdepth {\fb}M$. Notice that,  by Remark \ref{propinf}, $\fdepth {\fb}M<\infty$ if and only if $M/{\fb M}$ has finite length.

\begin{remark}\label{remark}%\begin{enumerate}[{\rm(i)}]
            %  \item
             The following equalities follows immediately from  Theorem \ref{filter1}.
\begin{align*}
\fgrad {\fa}{\ann N}M&= \inf \{i\in\mathbb N_0 |\,\supp \ext iRNM\nsubseteq\ V(\fa)\}, \\
         \fgrad {\fa}{\fb+\ann N}M &=\inf \{i\in\mathbb N_0 |\,\supp  \gh i{\fb}NM\nsubseteq V(\fa)\}.
       \end{align*}
       In particular,
       \begin{align*}
        \fgrad {\fa}{\fb}M&=\inf \{i\in\mathbb N_0 |\,\supp \ext iR{{R}/{\fb}}M\nsubseteq V(\fa)\}\\
   &= \inf \{i\in\mathbb N_0 |\,\supp \h i{\fb}M\nsubseteq\ V(\fa)\}.
           \end{align*}
          Suppose  in addition that $(R, \fm)$ is  local. Then
           \begin{align*}
            \fdepth {\fb}M&=\inf\{i\in\N_0|\, \dim \ext iR{R/{\fb}}M>0\}\\
            &=\inf\{i\in\N_0|\, \supp\h i{\fb}M\nsubseteq\{\fm\}\}.
           \end{align*}
            %  \item  Since $\fgrad R{\fb}M=\grad {\fb}M$, we have  the well known properties
                 %                  $$ \grad {\fb}M=\inf\{i\in\mathbb N_0 |\,\ext iR{{R}/{\fb}}M\neq 0\}
                %  =\inf\{i\in\mathbb N_0 |\,\h i{\fb}M\neq 0\}$$
                             %   and
                                  % $$\grad {\fb+\ann N}M= \inf\{i\in\mathbb N_0 |\,\gh i{\fb}NM\neq 0\}$$
%which are stated in \cite[Theorem 6.2.7]{bs} and \cite[Proposition 5.5]{b} respectively.
        %     \end{enumerate}
 \end{remark}
%-------------------------------------------------------------------------------------------------------
\section{A generalization of Nagel-Schenzel formula}
 Let $x_1,\ldots,x_n$ be an $\fa$-filter regular $M$-sequence in $\fa$.  Then, by  \cite[Proposition 1.2]{khs},
$$\h i{\fa}M=
\begin{cases}
\h i{(x_1,\ldots,x_n)}M   &\text{if } i<n\,,\\
\h {i-n}{\fa}{\h n{(x_1,\ldots,x_n)}M}  &\text{if } i\geq n\,.
\end{cases}$$
 This formula was first obtained by Nagel and Schenzel, in \cite[ Lemma 3.4]{NS}, in the case where $(R,\fm)$ is a local ring and $\fa=\fm$. Afterwards Khashyarmanesh, Yassi and Abbasi  \cite[Theorem 3.2]{kya} and  Mafi  \cite[Lemma 2.8]{ma1} generalized the second part of this formula for the generalized local cohomology modules as follows.

 Suppose that  $M$ has   finite projective dimension $d$ and that $x_1,\ldots,x_n$ is an $\fa$-filter regular $N$-sequence in $\fa$. Then
   $$\gh {i+n}{\fa}MN\cong\gh i{\fa}M{\h n{(x_1,\ldots,x_n)}N}$$ for all $i\geq d$.

  The following Theorem establishes the Nagel-Schenzel  formula for the generalized local cohomology modules. The first part of the following theorem  is needed in the proof of the Corollary \ref{corf1}.
  \begin{theorem}\label{sch}
  Let  $x_1,\ldots,x_n$ be an $\fa$-filter regular $N$-sequence in $\fa$. Then the following statements hold.
  \begin{enumerate}[{\rm(i)}]
    \item\label{sch1} $\gh i{\fa}MN\cong\gh i{(x_1,\ldots,x_n)}MN$ for all $i<n$.\\
    \item\label{sch2} If $\pd M =d <\infty$ and $L$ is projective, then
                $$\gh {i+n}{\fa}{\tansor MRL}N\cong\gh i{\fa}M{\gh n{(x_1,\ldots,x_n)}LN}$$
      for all $i\geq d$.
    \end{enumerate}
  \end{theorem}
  \begin{proof}
    (i)  Set $\boldsymbol{x}=x_1,\ldots,x_n$. Since $\gam{\fa}N\subseteq\gam{(\boldsymbol{x})}N$ we have a natural monomorphism $\varphi_{M,N}:\gh 0{\fa}MN\rightarrow\gh 0{(\boldsymbol{x})}MN$.
    Now, let ${\mu}_i({\fp}, N)$ be the $i$-th Bass number of $N$ with respect to a prime ideal $\fp$ and let $0\longrightarrow E^0\stackrel{d^0}{\longrightarrow}E^1\stackrel{d^1}{\longrightarrow}E^2\longrightarrow\cdots$ be the minimal injective resolution of $N$. Then,  by Proposition \ref{kh1}, ${\mu}_i({\fp}, N)=0$ for all $\fp\in\supp N\cap V(\boldsymbol{x})\setminus V(\fa)$ and all $i<n$. So          \begin{align*}
          \gam {\fa} {E^i}&=\bigoplus_{\fp\in\supp N\cap V(\fa)}E({R}/{\fp})^{{\mu}_i({\fp}, N)}\\&=\bigoplus_{\fp\in\supp N\cap V(\boldsymbol{x})}E({R}/{\fp})^{{\mu}_i({\fp}, N)}=\gam {(\boldsymbol {x})}{E^i}
\end{align*}
 for all $i<n$. Therefore $\varphi_{M,E^i}$ is an isomorphism for all $i<n$.
       Now let $i<n$. Since $\varphi_{M,{E^{i-1}}}$ and $\varphi_{M,{E^i}}$ are isomorphisms and $\varphi_{M,{E^{i+1}}}$ is a monomorphism, one can use  the following  commutative diagram
\begin{displaymath}
\xymatrix{
\gh 0{\fa}M{E^{i-1}}\ar[r]\ar[d]^{\varphi_{M,{E^{i-1}}}}&
\gh 0{\fa}M{E^{i}}\ar[r]\ar[d]^{\varphi_{M,{E^{i}}}}&
\gh 0{\fa}M{E^{i+1}}\ar[d]^{\varphi_{M,{E^{i+1}}}}\\
\gh 0{(\boldsymbol{x})}M{E^{i-1}}\ar[r]&
\gh 0{(\boldsymbol{x})}M{E^{i}}\ar[r]&
\gh 0{(\boldsymbol{x})}M{E^{i+1}}
\,}
\end{displaymath}
to see that the induced homomorphism
 $$\bar{\varphi}_{M,{E^i}}:\gh i{\fa}MN=\frac{\ker {\gh 0{\fa}M{d^{i}}}}{\im {\gh 0{\fa}M{d^{i-1}}}}\rightarrow\frac{\ker {\gh 0 {(\boldsymbol{x})}M{d^{i}}}}{\im {\gh 0{(\boldsymbol{x})}M{d^{i-1}}}}=\gh i{(\boldsymbol{x})}MN\,,$$
 is an isomorphism.

(ii) Set $F({\boldsymbol{\cdot}})=\gh 0{\fa}M{\boldsymbol{\cdot}}$ and $G(\boldsymbol{\cdot})=\gh 0{(\boldsymbol{x})}L{\boldsymbol{\cdot}}$. Then $F$ and $G$ are  left exact functors and $FG(\boldsymbol{\cdot})\cong\gh 0{\fa}{\tansor MRL}{\boldsymbol{\cdot}}$. Furthermore if $E$ is an injective $R$-module and $\textsf{R}^p F$ ($p\in \mathbb N_0$) is the $p$ -th right derived functor of $F$, then it follows  from  \cite[Lemma 1.1]{yks} and(i) that
\begin{align*}
   \textsf{R}^p F(G(E))&=\gh p{\fa}M{\gh 0{(\boldsymbol{x})}LE}
   \cong\gh p{\fa}M{\gh 0{\fa}LE}\\&\cong \ext pRM{\Hom RL{\gam {\fa}E}}=0
\end{align*}
for all $p\geq1$.
This yields the following spectral sequence
$$E^{p,q}_2=\gh p{\fa}M{\gh q{(\boldsymbol{x})}LN}\underset{p}{\Longrightarrow}\gh {p+q}{\fa}{\tansor MRL}N$$
(see for example \cite[Theorem 11.38]{r} ). Let $t=p+q\geq d+n$. If $q>n$, then $\h q{(\boldsymbol x)}N=0$  by \cite[Corollary 3.3.3]{bs}. Since $L$ is projective, it therefore follows that $\gh q{(\boldsymbol x)}LN=0$.  On the other hand if $q<n$, then $p>d=\pd M$. Hence
\begin{align*}
    E_2^{p,q}=\gh p{\fa}M{\gh q{(\boldsymbol{x})}LN}\cong\gh p{\fa}M{\gh q{\fa}LN}\cong\ext pRM{\gh q{\fa}LN}=0.
\end{align*}
Therefore, for $t\geq n+d$, there is a  collapsing on the line $\fq=n$. Thus, there are isomorphisms
$$\gh {t-n}{\fa}M{\gh n{(\boldsymbol{x})}LN}\cong\gh t{\fa}{\tansor MRL}N$$
for all $t\geq n+d$.
    \end{proof}
  %{  The following corollary is an immediate consequence of Theorem \ref{sch}.
 %\begin{corollary}[see\cite{NS}  Lemma 3.4 and \cite{khs} Proposition 1.2]
 %    Let $\boldsymbol {x}=x_1,\ldots,x_n$ be an $\fa$-filter regular $N$-sequence in $\fa$. Then we have
%$$\h i{\fa}M=
%\begin{cases}
%\h i{(\boldsymbol {x})}M   &\text{if } i<n\,,\\
%\h {i-n}{\fa}{\h n{(\boldsymbol {x})}M}  &\text{if } i\geq n\,,
%\end{cases}$$
 % which is known as Nagel-Schenzel formula.
 % \end{corollary}
 %-------------------------------------------------------------------------------------------------------
\section{ Artinianness of generalized local cohomology modules}
    Let $(R,\fm)$ be a Noetherian local ring. %In Definition \ref{filter2}  we introduce the concept of  $\fdepth {\fa}M$.
    In view of   \cite[Theorem 3.1]{me} and \cite[Theorem 3.10]{lt}, one can see that $\fdepth {\fa}M$ is  the least integer $i$ for which $\h i{\fa}M$ is not Artinian. Also,  as a main result, it was proved in \cite[Theorem 2.2]{ct} that  $\fdepth {\fa+\ann M}N$ is the least integer $i$ such that $\gh i{\fa}MN$ is not Artinian. We use rather a short argument to generalize this to the  case in which $R$ is not necessarily  a local ring. The following lemma is elementary.
\begin{lemma}[\cite{s} Exercise 8.49]\label{lemart}
    Let $X$ be an Artinian $R$-module, then $\ass X=\supp  X$ is a finite subset of $ \max (R)$.
\end{lemma}

\begin{theorem}\label{art} Let $\mathcal M$ be the set of all finite subsets of $\max(R)$. Then
   \begin{align*}
\sup_{A\in\mathcal M} \fgrad {\cap_{\fm\in A} \fm}{&\fa+\ann M}N\\
&=\inf\{i\in\mathbb N_0|\, \gh i{\fa}MN \, \textrm{is not Artinian}\}\\
    &=\inf \{i\in\mathbb N_0|\, \supp \gh i{\fa}MN \nsubseteq \max (R) \}\\
   &=\inf \{i\in\mathbb N_0|\, \supp \gh i{\fa}MN \nsubseteq A \, \textrm{ for all } A\in\mathcal M\}
        \end{align*}
\end{theorem}
\begin{proof}
      Since $\gh i{\fa}MN\cong\gh i{\fa+\ann M}MN$, we can assume that $\ann M\subseteq \fa$.
      It is clear that
            $$ \sup_{A\in\mathcal M} \fgrad {\cap_{\fm\in A} \fm}{\fa}N=\inf \{i\in\mathbb N_0|\, \supp \gh i{\fa}MN \nsubseteq A \, \textrm{ for all } A\in\mathcal M\}.         $$
Let $\mathcal S$  be either $\{X\in {\mathcal C_R}|\, \supp X\subseteq \max (R)\}$ or $\{X\in {\mathcal C_R}|\, \supp X\subseteq A \, \textrm{for some }\\ A\in\mathcal M\} $, where $\mathcal C_R$  is the category of $R$-modules.
      Set $r=\inf\{i\in\mathbb N_0|\, \gh i{\fa}MN \,\\ \textrm{is not Artinian}\}$ and $s=\inf \{i\in\mathbb N_0|\, \gh i{\fa}MN \notin \mathcal S \}$.
       %Note that $r, s \in\mathbb N_0\cup\{\infty\}$.
 By Lemma \ref{lemart}, $r\leq s$. If $r=\infty$, there is noting to prove. Assume that $r<\infty$. We show  by induction on $r$, that $\gh r{\fa}MN \notin \mathcal S $.

     If $r=0$, then $ {\gh 0{\fa}MN}\notin \mathcal S $. Now suppose, inductively, that $r>0$ and that the result has been proved for smaller values of $r$. In view of \cite[Lemma 1.1]{yks} the exact sequence
 $$0\longrightarrow\gam {\fa}N\longrightarrow N\longrightarrow{N}/{\gam {\fa}N}\longrightarrow 0$$
 induces the following  long exact sequence
  \begin{align*}
 \cdots&\longrightarrow\ext iRM{\gam {\fa}N}\longrightarrow \gh i{\fa}MN\longrightarrow\gh i{\fa}M{{N}/{\gam {\fa}N}}\\
  &\longrightarrow\ext {i+1}RM{\gam {\fa}N}\longrightarrow\cdots.
    \end{align*}
 Since $\gh 0{\fa}MN$  has  finite length, we have
 $$\supp\gh 0{\fa}MN=\ass  \Hom RM{\gam {\fa}N}=\ass {\gam {\fa}N};$$
   so that ${\gam {\fa}N}\in\mathcal S$. Thus $\ext iRM{\gam {\fa}N}\in\mathcal S$ for all $i\in\mathbb N_0$. It follows that for each $i\in\mathbb N_0$, $\gh i{\fa}MN\in\mathcal S$ if and only if $\gh i{\fa}M{ {N}/{\gam {\fa}N}}\in\mathcal S$. Also we have $\gh i{\fa}MN$  is Artinian if and only if $\gh i{\fa}M{ {N}/{\gam {\fa}N}}$ is Artinian. Hence  we can replace $N$ by ${N}/{\gam {\fa}N}$ and assume that $N$ is an $\fa$-torsion free $R$-module. Thus there exists an element $x\in \fa$ which is a non-zero divisor on  $N$.
The exact sequence
$$0\longrightarrow N\stackrel{x}{\longrightarrow}  N\longrightarrow  N/{xN}\longrightarrow 0$$
induces the long exact sequence
$$\cdots\longrightarrow\gh i{\fa}MN\stackrel{x}{\longrightarrow} \gh i{\fa}MN\stackrel{f_i}{\longrightarrow}\gh i{\fa}M{{N}/{xN}}\longrightarrow\gh {i+1}RMN\longrightarrow\cdots.$$
Since $\gh i{\fa}MN$ is  Artinian for all $i<r$, we may use  the above sequence to see that $\gh i{\fa}M{ {N}/{xN}}$ is  Artinian for all $i<r-1$. On the other hand, $\gh r{\fa}MN$ is not Artinian. Hence, using the above exact sequence and  \cite[ Theorem 7.1.2]{bs}, we see that   $0:_{\gh r{\fa}MN}x\cong {\gh {r-1}{\fa}M{ {N}/{xN}}}/\im {f_{r-1}}$ is not Artinian.   Thus $\gh {r-1}{\fa}M{{N}/{xN}}$ is not Artinian; and hence,  by inductive hypothesis, $\gh {r-1}{\fa}M{{N}/{xN}}\notin \mathcal S $.  So, again by using the above sequence,  we get $ \gh r{\fa}MN\notin \mathcal S $. This  completes  the inductive step.
   \end{proof}
\begin{corollary}\label{corart}
    Suppose that $\supp L=\supp {M}/{\fa M}$. Then
    \begin{align*}
        \inf \{i\in\mathbb N_0|\, \gh i{\fa}MN \,\textrm{is not  Artinian}&\}
         =\inf\{i\in\mathbb N_0|\, \dim \ext iRLN>0\, \}.
    \end{align*}
    \end{corollary}
    \begin{proof}
              Let $n\in\mathbb N_0$. Then, by the  Theorem \ref{art}, $\gh i{\fa}MN$ is an Artinian $R$-module for all $i\leq n$ if and only if $n<\fgrad {\fm_1\cap\ldots\cap\fm_t}{\fa+\ann M}N$ for some maximal ideals $\fm_1,\ldots,\fm_t$ of $R$.  By the Remarks \ref{remark} (i), it is equivalent to  $\supp \ext iRLN\subseteq \{\fm_1,\ldots,\fm_t\}$ for some maximal ideals $\fm_1,\ldots,\fm_t$ of $R$ and for all $i\leq n$. This proves the assertion.
    \end{proof}

  The following corollary extend the  main result of  \cite{ta} to  the generalized local cohomology modules.
\begin{corollary}\label{corart2}
Let $n\in\mathbb N$. Then $\gh i{\fa}MN$ is Artinian for all $i<n$ if and only if $\gh i{\fa R_\fp}{M_\fp}{N_\fp}$ is Artinian for all $i<n$ and all prime ideal $\fp$.
\end{corollary}
\begin{proof}
This is immediate by the Corollary \ref{corart}.
\end{proof}
\begin{corollary}\label{corf1}
    Let $\overline{R}={R}/({\fa+\ann M+\ann N})$  be a semi local ring and let $\fr$ be the inverse image of the Jacobson radical of  $\overline R$ in $R$. Then we have
    \begin{align*}
   \fgrad {\fr}{\fa+\ann M}N=&\inf\{i\in\mathbb N_0|\, \gh i{\fa}MN \, \textrm{is not Artinian}\}\\
    =&\inf\{i\in\mathbb N_0|\, \gh i{\fa}MN\ncong \gh i{\fr}MN\}
   \end{align*}
     \end{corollary}
\begin{proof}
    The first equality is immediate by  Theorem \ref{art}. To prove the second equality, let $n\leq\fgrad {\fr}{\fa+\ann M}N$ and let $x_1,\ldots,x_n$ be an $\fr$-filter regular $N$-sequence in $\fa+\ann M$. Then $x_1,\ldots,x_n$ is an $\fa+\ann M$-filter regular $N$-sequence. So  by Theorem \ref{sch}(i),
$$\gh i{\fa}MN\cong\gh i{\fa+\ann M}MN\cong\gh i{(x_1,\ldots,x_n)}MN\cong\gh i{\fr}MN$$
for all $i<n$.  If $\fgrad {\fr}{\fa+\ann M}N=\infty$, then  the above argument shows that,  $\inf\{i\in\mathbb N_0|\, \gh i{\fa}MN\ncong \gh i{\fr}MN\}=\infty$ and therefore the required equality  holds. Therefore, we may assume that $\fgrad {\fr}{\fa+\ann M}N=n<\infty $.  By the first equality, $\gh n{\fa}MN$ is not Artinian while $\gh n{\fr}MN$ is Artinian. Hence the second   equality  holds.
\end{proof}
It was shown  in \cite[Theorem 2.2]{z} that if $\dim R/\fa=0$, then $\gh i{\fa}MN$ is Artinian  for all $i\in\mathbb N_0$. The following corollary is a generalization of this.
\begin{corollary}Let $\overline{R}={R}/({\fa+\ann M+\ann N})$. Then
 $\gh i{\fa}MN$  is an Artinian $R$-module for all $i\in\mathbb N_0$ if and only if  $\dim\overline{R}=0$. In particular,
 $\ext iRMN$ has finite length for all $i\in\mathbb N_0$ if and only if $\dim R/(\ann M+\ann N)=0$.
\end{corollary}
\begin{proof}
Assume that $\fp$ is a prime ideal of $R$. By the Corollary \ref{corf1}, $\gh i{\fa R_\fp}{M_\fp}{N_\fp}$ is Artinian for all $i<n$ if and only if
 $\fdepth {(\fa+\ann M)R_\fp}{N_\fp}=\infty$ or equivalently $\dim_{R_\fp}{N_\fp /{(\fa R_\fp+(\ann M) R_\fp)N_\fp}}=0$  (Remark \ref{propinf}). Now,  the result follows by corollary \ref{corart2}.
\end{proof}
%-----------------------------------------------------------------------------------------------
\section{Attached primes of the top generalized local cohomology modules}
Let $X\neq 0$ be an $R$-module. If,  for every $x\in R$, the endomorphism on $X$ given by multiplication by $x$ is either nilpotent or   surjective, then $\fp=\sqrt {\ann X}$  is prime and $X$ is called a $\fp$-secondary  $R$-module. If for some secondary submodules $X_1,\ldots,X_n$ of $X$ we have  $X=X_1+\ldots+X_n$, then we say that $X$ has a secondary representation. One may assume that the prime ideals  $\fp_i=\sqrt{\ann X_i}$, $i=1,\ldots,n$, are distinct and, by omitting redundant   summands, that the representation is minimal. Then the set $\att X=\{\fp_1,\ldots,\fp_n\}$  does not depend on the choice of a minimal secondary representation of $X$. Every element of  $\att X$  is called an attached prime ideal of $X$. It is well known that an Artinian $R$-module has a  secondary representation. The reader is referred to  \cite{m} for more information about the  theory of secondary representation.

Let $(R, \fm)$ be a local ring and $n=\dim N<\infty $ and $d=\pd M<\infty$. It  was proved in \cite[Theorem 2.3]{gc}  that $\gh{n+d}{\fa}MN$ is Artinian and that
 $$\att\gh{n+d}{\fa}MN=\{\fp\in\ass  N | \cd{\fa}M{ R/{\fp}}=n+d\},$$
 where, for an $R$-module $Y$, $\cd{\fa}MY$ is the greatest integer $i$ such that $\gh i{\fa}MY\neq 0$. Notice that $\cd {\fa}MN\leq d+n$ \cite[Lemma 5.1]{b}.  Next, we prove the above result without the local  assumption on  $R$.
%In this section, we give a formula for the top generalized local cohomology modules.
The following  lemmas are needed.
\begin{lemma}[  \cite{an} Theorem A and B ]\label{att1}
    Let $\pd M<\infty$. Then
    \begin{enumerate}[\upshape (i)]
      \item  $\cd {\fa}MN\leq\cd {\fa}ML$ whenever  $\supp N\subseteq\supp L$.
      \item $\cd {\fa}ML=\max\{\cd {\fa}MN, \cd {\fa}MK\}$ whenever   $0\rightarrow N\rightarrow L\rightarrow K\rightarrow 0$ is an exact sequence.
    \end{enumerate}
\end{lemma}

\begin{lemma}\label{att2}
Let $\pd M<\infty$, $\dim N<\infty$, $t=\cd {\fa}MN\geq0$ and $$\Sigma=\{L\subsetneqq N| \cd {\fa}ML <t\}.$$ Then $\Sigma$ has the largest element with respect to inclusion, $L$
 say, and the following statements hold.
\begin{enumerate}[\upshape (i)]
  \item If $K$ is  a non-zero submodule  of ${N}/{L}$, then $\cd {\fa}MK=t$.
  \item $\gh t{\fa}MN\cong\gh t{\fa}M{{N}/{L}}$.
  \item If $t=\pd M+\dim  N$, then $$\ass  N/L=\{\fp\in\ass N | \cd{\fa}M{ R/{\fp}}=t \}.$$
\end{enumerate}
\end{lemma}
\begin{proof}
     Since $N$ is Noetherian, $\Sigma$ has a maximal element, say $L$. Now  assume that $L_1, L_2$ are  elements  of $\Sigma$.  Using the exact sequence $$0\rightarrow L_1\cap L_2\rightarrow L_1\oplus L_2\rightarrow L_1+L_2\rightarrow 0$$ and  Lemma \ref{att1} we see that  $t> \cd {\fa}M{L_1+L_2}$.
    Hence  the sum of any two elements of $\Sigma$ is again in $\Sigma$. It follows that $L$ contains every element of $\Sigma$; and so it is the largest one.

    (i) Let $K= {K'}/L$ be a non-zero submodule of $ N/L$. Since $L$ is the largest element of $\Sigma$, by  applying  Lemma \ref{att1} to the exact sequence $$0\rightarrow L\rightarrow K'\rightarrow K\rightarrow 0$$  we see that $t=\cd {\fa}MK$.

    (ii) The exact sequence $0\rightarrow L\rightarrow N\rightarrow  N/L\rightarrow 0$ induces the exact sequence $$0=\gh t{\fa}ML\rightarrow \gh t{\fa}MN\rightarrow \gh t{\fa}M{ N/L}\rightarrow \gh {t+1}{\fa}ML=0.$$
    Thus $\gh t{\fa}MN\cong\gh t{\fa}M{{N}/{L}}$.

    (iii) Assume that $\cd{\fa}MN=\pd M+\dim N$.  For each $\fp$ in $\ass L$, we have $\cd {\fa}M{ R/{\fp}}<t$; so that  $$ \{\fp\in\ass N | \cd{\fa}M{ R/{\fp}}=t \}\subseteq\ass  N/L.$$
     To establish the reverse inclusion, let $\fp\in\ass {N}/{L}$.  Then by (i) and \cite[Lemma 5.1]{b}  $t=\pd M+\dim { R/{\fp}}$. Therefore $\fp\in\ass  N$ and equality holds.
         \end{proof}
 \begin{theorem}\label{att3}
    Let $d=\pd  M<\infty$ and $n=\dim N<\infty$. Then the $R$-module $\gh {n+d}{\fa}MN$  is Artinian and
    $$\att\gh{n+d}{\fa}MN=\{\fp\in\ass  N | \cd{\fa}M{ R/{\fp}}=n+d\}.$$
 \end{theorem}
 \begin{proof}
 Let $\boldsymbol x=x_1,\ldots, x_n$ be an $\fa$-filter regular $N$-sequence in $\fa$ and let $E^\bullet$ be the minimal injective resolution of $\h n{(\boldsymbol x)}N$. Since, by \cite[Exercise 7.1.7]{bs}, $\h n{(\boldsymbol x)}N$ is Artinian, every component of $E^\bullet$ is Artinian. On the other hand by  \ref{sch}
  $$\gh {n+d}{\fa}MN\cong\gh{d}{\fa}M{\h n{(\boldsymbol x)}N}\cong H^d(\Hom RM{\gam {\fa}{E^\bullet}}).$$
  It follows that  $\gh {n+d}{\fa}MN$ is Artinian.

  Now we prove that $\att\gh{n+d}{\fa}MN=\{\fp\in\ass  N | \cd{\fa}M{ R/{\fp}}=n+d\}$.  If $\cd {\fa}MN<n+d$, then $\att\gh{n+d}{\fa}MN=\emptyset=\{\fp\in\ass  N | \cd{\fa}M{ R/{\fp}}=n+d\}$. So one can assume that $t=\cd {\fa}MN=n+d$. Let $L$ be the largest submodule of $N$ such that $\cd {\fa}ML<t$. By Lemma \ref{att2}, there is no non-zero submodule $K$ of $ N/L$ such that $\cd {\fa}MK<t$. Also  we have  $\gh t{\fa}M{N}\cong\gh t{\fa}M{ N/L}$ and $\ass  N/L = \{\fp\in\ass N | \cd {\fa}M{ {R}/{\fp}}=t \}$. Moreover $t=\cd {\fa}M{ {N}/{L}}=\pd M +\dim N/L$.
  Thus we may replace $N$ by $ N/L$ and  prove that $\att \gh t{\fa}MN=\ass N$. Now, for any non-zero submodule $K$ of $N$, $\cd {\fa}MK=t$ and $\dim K=n$.

   Assume that $\fp\in\att \gh t{\fa}MN$. We have $\fp\supseteq\ann \gh t{\fa}MN\supseteq\ann N$. Hence $\fp\in\supp N$. Now Let $x\in R\setminus\bigcup_{\fp\in \ass N}\fp$. The exact sequence $$0\longrightarrow N\stackrel{x}{\longrightarrow} N\longrightarrow  {N}/{xN}\longrightarrow 0$$ induces the exact sequence $$\gh t{\fa}MN\stackrel {x}\rightarrow{\gh t{\fa}MN}\rightarrow\gh t{\fa}M{ N/{xN}}=0.$$
   Therefore $x\notin \bigcup_{\fp\in \att \gh t{\fa}MN}\fp$. So $\bigcup_{\fp\in \att \gh t{\fa}MN}\fp\subseteq\bigcup_{\fp\in \ass N}\fp$. Thus  $\fp\subseteq\fq$ for some $\fq\in\ass N$. Hence $\fp=\fq$ and $\att\gh t{\fa}MN\subseteq\ass N$.
  Next we prove the reverse inclusion. Let $\fp\in\ass N$ and let $T$ be  a $\fp$-primary submodule of $N$. We have  $t=\cd {\fa}M{ R/{\fp}}=\cd {\fa}M{ N/T}$.
   Moreover $ N/T$ has no non-zero submodule $K$ such that $\cd {\fa}MK<t$. Hence, using  the above argument, one can  show that $\att \gh t{\fa}M{ N/T}\subseteq\ass  N/T=\{\fp\}$. It follows that
  $$\{\fp\}=\att\gh t{\fa}M{ N/T}\subseteq \att\gh t{\fa}MN.$$
  This completes  the proof.
  \end{proof}
  \begin{corollary}\label{coratt}
   Let $d=\pd  M<\infty$ and $n=\dim N<\infty$. Then
   $$\att \gh {n+d}{\fa}MN\subseteq\supp M\cap\att \h n{\fa}N.$$
  \end{corollary}
  \begin{proof}
  If $\att \gh {n+d}{\fa}MN=\emptyset$, there is nothing to prove. Assume that $\fp\in\att  \gh {n+d}{\fa}MN$. Then, by \ref{att3}, $\fp \in\ass N$ and $\gh {n+d}{\fa}M{R/\fp}\neq 0$. Next one can use the spectral sequence
  $$E^{p,q}_2=\ext pRM{\h q{\fa}{R/\fp}}\underset{p}{\Longrightarrow}\gh {p+q}{\fa}M{R/\fp}$$
  to see that  $\gh {n+d}{\fa}M{R/\fp}\cong\ext dRM{\h n{\fa}{R/\fp}}$. Therefore $\h n{\fa}{R/\fp}\neq 0$; and hence $\cc {\fa}{R/\fp}=n$. Thus, again by \ref{att3}, $\fp\in\att \h n{\fa}N$.  Also, we have $\fp\supseteq\ann \ext dRM{\h n{\fa}N}\supseteq\ann M$, which completes the proof.
  \end{proof}

Let $X$ be an $R$-module.   Set $E=\bigoplus_{\fm\in\max R}E(R/\fm)$ (minimal injective cogenerator of $R$) and $D=\Hom R\cdot E$. We note that the canonical map $X\rightarrow DDX$ is an injection. If this map is an isomorphism we say that $X$ is (Matlis) reflexive. The following lemma yields a classification of modules which are reflexive with respect to $E$.
\begin{lemma}[\cite{ber} Theorem 12]\label{enochs}
An $R$-module $X$ is reflexive if and only if it has a finite submodule $S$ such that $X/S$ is artinian and that $R/{Ann X}$ is a complete semilocal ring.
\end{lemma}
Assume that $\fa\subseteq\fb$ and $R/\fa$ is a complete semilocal ring. By above lemma $R/\fa$ is  reflexive as an $R$-module. On the other hand, the category of reflexive $R$-modules is a Serre subcategory of the category of $R$-modules. Therefore $R/\fb$ is reflexive as an $R$-module and hence, by the  above lemma,  $R/\fb$ is a complete semilocal ring. We shall use the conclusion of this discussion in the proof of the next theorem.
\begin{theorem}\label{att6}
Let $M, N$ be two finite $R$-modules with $\pd M=d<\infty$ and $\dim N=n<\infty$. Let $\fb=Ann \h {n}{\fa}N$. If $R/\fb$ is a complete semilocal ring, then
$$\att \gh {n+d}{\fa}MN=\supp  {\ext dRMR}\cap\att\h n{\fa}N.$$
In particular, if in addition,  $\pd_{R_\fp}M_\fp=\pd M$ for all $\fp\in\supp M$, then
$$\att \gh {n+d}{\fa}MN=\supp  M\cap\att\h n{\fa}N.$$
\end{theorem}
\begin{proof}
Since $\ext dRM\cdot$ is a right exact $R$-linear covariant functor,we have
$$\gh {n+d}{\fa}MN\cong\ext dRM{\h n{\fa}N}\cong\ext dRMR\otimes_R\h n{\fa}N.$$
Set $\fc=\ann \gh {n+d}{\fa}MN$.  It is clear that $\fb\subseteq\fc$. Therefore $R/\fc$ is a complete semilocal ring. Now, by Lemma \ref{enochs}, \cite[Exercise 7.2.10]{bs} and \cite [VI.1.4 Proposition 10]{bo} we have
\begin{align*}
    \att\gh {n+d}{\fa}MN&=\att DD\gh {n+d}{\fa}MN\\&=\ass D\gh {n+d}{\fa}MN\\&=\ass D(\ext dRMR\otimes_R\h n{\fa}N)\\&=\ass \Hom  R{\ext dRMR}{D\h n{\fa}N}\\&=\supp {\ext dRMR}\cap\ass D\h n{\fa}N\\&=\supp {\ext dRMR}\cap\att DD\h n{\fa}N\\&=\supp {\ext dRMR}\cap\att \h n{\fa}N
\end{align*}
The final assertion follows immediately from the first  equality, \cite[ Lemma 19.1(iii)]{mat} and the fact that
 $ \supp {\ext dRMR}\subseteq \supp M$.
\end{proof}

 By  Corollary \ref{coratt} $\att \gh {n+d}{\fa}MN\subseteq\att \h n{\fa}N$. Next, we give an example to show that this inclusion  may be strict even if  $(R, \fm)$ is a complete regular local ring and $\fa=\fm$. Also, this  example shows that the following theorem of Mafi  is not true.

  \cite[Theorem 2.1]{ma2}:  Let $(R, \fm)$ be a commutative Notherian local ring and $n=\dim N$, $d=\pd  M<\infty$. If $\gh {n+d}{\fm}MN\neq 0$, then
  $$\att \gh {n+d}{\fm}MN=\att \h n{\fm}N.$$
\begin{example}
Let  $(R, \fm)$ be a complete regular local ring of a dimension $n\geq 2$ and assume that $R$ has two  distinct prime ideals $\fp, \fq$ such that $\dim R/\fp=\dim R/\fq=1$. Set $M=R/\fp$ and $N=R/\fp\oplus R/\fq$. Then, by  Theorem \ref{att3},
$$\att \h 1{\fm}N=\{\fp, \fq\}.$$
On the other hand, $\pd M=\dim R -\depth M=n-1$ and $\dim N=1$. Now, by Theorem \ref{att6},
$$\att \gh n{\fm}MN=\supp M\cap\att\h 1{\fm}N=\{\fp\}.$$
    Therefore \cite[Theorem 2.1]{ma2} is not true. Also, by \cite[Proposition 7.2.11]{bs},
  $$\surd(Ann \gh n{\fm}MN=\bigcap_{\fp\in\att \gh n{\fm}MN}\fp=\fp$$ and $$\surd(Ann \h 1{\fm}N)=\bigcap_{\fp\in\att \h 1{\fm}N}\fp=\fp\cap\fq.$$
Hence, again, Corollary 2.2 and Corollary 2.3 of  \cite{ma2}  are not true. We note that, the other results of \cite{ma2} are concluded from \cite[Theorem 2.1, Corollary 2.2 and Corollary 2.3]{ma2}.
    \end{example}

It is known that if $(R, \fm)$ is a local ring and $\dim M=n>0$, then $\h n{\fm}M$ is not finite \cite[Corollary 7.3.3]{bs}. It was proved in  \cite[Proposition 2.6]{gc}  that if $d=\pd M<\infty$ and $0<n=\dim N$, then  $\gh {n+d}{\fm}MN$ is not finite whenever it is non-zero. Next, we provide a generalization of this result. The following lemma, which is needed in the proof of the next proposition, is elementary.
  \begin{lemma}\label{att4}
    Let $X$ be an $R$-module. Then $X$ has finite length if and only if  $X$ is Artinian and $\att X \subseteq \max R$. Moreover if $X$ has finite length, then $\att X= \supp X=\ass X$.
  \end{lemma}
  \begin{proposition}\label{att5}
  Let $d=\pd M<\infty$, $0<n=\dim N<\infty$. If  $\gh {n+d}{\fa}MN\neq0$, then it is not finite.
  \end{proposition}
  \begin{proof}
   Assume that $\fp\in\att \gh {n+d}{\fa}MN$. By \ref{att3},   $\gh {n+d}{\fa}MN$  is an Artinian $R$-module and $n+d=\cd {\fa}M{ R/{\fp}}= \pd M+\dim { R/{\fp}}$. Therefore $\dim R/{\fp}=n>0$;  So that  $\att \gh {n+d}{\fa}MN\varsubsetneq\max R$. It follows that, in view of \ref{att4}, $\gh {n+d}{\fa}MN$ is not finite.
  \end{proof}

  %-----------------------------------------------------------------


\begin{thebibliography}{99}
\bibitem{an}
J. Amjadi and R. Naghipour,
Cohomological dimension of generalized local cohomology modules, {\it Algebra Colloq.} {\bf 15} (2008), no. 2, 303--308. 
 \bibitem{ber}
 R. G. Belshoff, E. E. Enochs and J. R. Garc\'{i}a Rozas,
Generalized Matlis duality, {\it Proc. Amer. Math. Soc.} {\bf 1}28 (2000), no. 5, 1307--1312.
  \bibitem{b}
 M. H. Bijan-Zadeh,
 A common generalization of local cohomology theories,
  {\it Glasgow Math. J.} {\bf 21} (1980), no. 2, 173--181.
\bibitem{bo}
 N. Bourbaki,  
 {\it Commutative Algebra}, Chapter 1-7, Elements of Mathematics, Springer-Verlage, Berlin, 1998.
\bibitem{bs}
M. P. Brodmann and R. Y.  Sharp,
{\it Local Cohomology: an Algebraic introduction with geometric applications},
Cambridge Studies in Advanced Mathematics,  60.
 Cambridge University Press, Cambridge, 1998.
 \bibitem{bh}
W. Bruns and J. Herzog,
 {\it Cohen-Macaulay rings},
 Cambridge Studies in Advanced Mathematics,  39. Cambridge University Press, Cambridge, 1993.
   \bibitem{ct}
 L. Chu and Z. Tang,
 On the Artinianness of generalized local cohomology, 
 {\it Comm. Algebra} {\bf 35} (2007), no. 12, 3821--3827.

 \bibitem{dy}
M. T.  Dibaei and S. Yassemi,
Attached primes of the top local cohomology modules with respect to an ideal,
 {\it Arch. Math. (Basel)} {\bf 84} (2005), no. 4, 292--297.

\bibitem{gc}
 Y. Gu and L. Chu,
 Attached primes of the top generalized local cohomology modules, 
 {\it Bull. Aust. Math. Soc.} {\bf 79} (2009), no. 1, 59--67. 
 \bibitem{h}
J. Herzog,
 Komplexe, Aufl\"{o}sungen und Dualit\"at in der lokalen Algebra,
Habilitationsschrift, Universit\"at Regensburg, 1970.
 \bibitem{khs}
  K.  Khashyarmanesh and Sh. Salarian,
 Filter regular sequences and the finiteness of local cohomology modules, 
 {\it Comm. Algebra} {\bf 26} (1998), no. 8, 2483--2490.
 \bibitem{kya}
K. Khashyarmanesh, M. Yassi and A. Abbasi,
 Filter regular sequences and generalized local cohomology modules,
 {\it Comm. Algebra} {\bf 32} (2004), no. 1, 253--259.
\bibitem{lt}
R. L\"{u} and Z. Tang,
 The $f$-depth of an ideal on a module, 
 {\it  Proc. Amer. Math. Soc.} {\bf 130} (2002), no. 7, 1905--1912 (electronic). 
\bibitem{m}
I. G. MacDonald,
Secondary representation of modules over a commutative ring,  Symposia Mathematica, Vol. XI (Convegno di Algebra Commutativa, INDAM, Rome, 1971), pp. 23--43. Academic Press, London, 1973.
\bibitem{ms}
I. G. Macdonald and R. Y. Sharp,
 An elementary proof of the non-vanishing of certain local cohomology modules, 
 {\it Quart. J. Math. Oxford Ser. (2)} {\bf 23} (1972), 197--204.
 \bibitem{ma1}
A. Mafi,
On the associated primes of generalized local cohomology modules, 
{\it  Comm. Algebra} {\bf 34} (2006), no. 7, 2489--2494.
\bibitem{ma2}
A. Mafi,
Top generalized local cohomology modules,
{\it Turkish J. Math.} {\bf 35} (2011), no. 4, 611--615.
 \bibitem{mat}
H.  Matsumura,
{\it Commutative ring theory}, Cambridge Studies in Advanced Mathematics, 8. Cambridge University Press, Cambridge, 1986.
\bibitem{me}
L. Melkersson,
 Some applications of a criterion for Artinianness of a module, 
 {\it J. Pure Appl. Algebra} {\bf 101} (1995), no. 3, 291--303. 
 \bibitem{NS}
 U.  Nagel and P. Schenzel,
 Cohomological annihilators and Castelnuovo-Mumford regularity, {\it Commutative algebra: syzygies, multiplicities, and birational algebra (South Hadley, MA, 1992)}, 307--328, Contemp. Math., 159, {\it Amer. Math. Soc., Providence, RI}, 1994.
\bibitem{r}
J. J. Rotman,
{\it An introduction to homological algebra},
 Pure and Applied Mathematics 85. Academic Press, Inc.,
New York, 1979.
\bibitem{s}
R. Y. Sharp,
{\it Steps in commutative algebra},
London Mathematical Society Student Texts,  19. Cambridge University Press, Cambridge, 1990.

\bibitem{sv}
J. St\"{u}ckrad and W. Vogel,
{\it Buchsbaum rings and applications. An interaction between algebra, geometry and topology}, Springer-Verlag, Berlin, 1986.
\bibitem{su}
N. Suzuki,
On the generalized local cohomology and its duality,  {\it J. Math. Kyoto Univ.} {\bf 18} (1978), no. 1, 71--85.
\bibitem{ta}
Z. Tang,
Local-global principle for the Artinianness of local cohomology modules, {\it Comm. Algebra} {\bf 40} (2012), no. 1, 58--63. 
\bibitem{t}
 N. V. Trung,
   Absolutely superficial sequences, 
   {\it  Math. Proc. Cambridge Philos. Soc.} {\bf 93} (1983), no. 1, 35--47.
\bibitem{yks}
S. Yassemi, L. Khatami and T. Sharif,
Associated primes of generalized local cohomology modules, 
{\it Comm. Algebra} {\bf 30} (2002), no. 1, 327--330.
 \bibitem{z}
 N. Zamani,
 On graded generalized local cohomology, {\it Arch. Math. (Basel)} {\bf 86} (2006), no. 4, 321--330.
\end{thebibliography}
\end{document}